\documentclass[12pt,twoside]{article}
\usepackage{amstext}
\usepackage{theorem}
\usepackage{amssymb}
\usepackage{latexsym}
\usepackage{amsmath}
\usepackage{amscd}

\usepackage{color}

\usepackage[
backend=biber,
style=alphabetic,
sorting=ynt
]{biblatex}
\newtheorem{theorem}{Theorem}[section]
\newtheorem{proposition}[theorem]{Proposition}
\newtheorem{lemma}[theorem]{Lemma}
\newtheorem{corollary}[theorem]{Corollary}
\newtheorem{claim}[theorem]{Claim}
\newtheorem{definition}[theorem]{Definition}

\newtheorem{fact}[theorem]{Fact}

\newcommand{\N}{\mathbb{N}}

\newcommand{\F}{\mathbb{F}}

\newcommand{\ord}{{\rm ord}}

\newenvironment{proof}{{\it Proof\,:}}{\hfill$\Diamond$}

\usepackage{hyperref}
\usepackage[
backend=biber,
style=alphabetic,
sorting=ynt
]{biblatex}
\usepackage{footmisc}
\DefineFNsymbols{mySymbols}{{\ensuremath\dagger}{\ensuremath\ddagger}\S\P
   *{**}{\ensuremath{\dagger\dagger}}{\ensuremath{\ddagger\ddagger}}}
\setfnsymbol{mySymbols}

\begin{document}

\author{Dimitra Chompitaki, 
Manos Kamarianakis\footnote{Corresponding Author,
\url{kamarianakis@uoc.gr}} \ and Thanases Pheidas\\
University of Crete, Dept. of Mathematics \& Applied Mathematics}

\title{Decidability of the theory of addition and the
Frobenius map in rings of rational functions}

\date{July 2021}

\maketitle

\begin{abstract} We prove model completeness for the theory of addition and the Frobenius map for certain subrings of rational functions in positive characteristic.  More precisely: 
Let $p$ be a prime number, $\F_{p}$ the prime 
field with $p$ elements, $F$ a field algebraic over $\F_p$ and $z$ a variable. 
We show that the structures of rings $R$, which are generated over $F[z]$ by 
adjoining a finite set of inverses of irreducible polynomials of $F[z]$
(e.g., $R=\F_p[z, \frac{1}{z}]$),  
with addition, the Frobenius map $x\mapsto x^p$ and the 
predicate `$\in F$' - together with function symbols and constants 
that allow building all elements of $\F_p[z]$ - are model complete, 
i.e., each formula is equivalent to an existential formula.  
Further, we show that in these structures  all questions, i.e., {\it first order sentences}, about the rings $R$ may be, constructively, translated into questions about $F$. 
\end{abstract}

\maketitle

\section{Introduction} 
Let $p$ be a prime number, $\F_p$ a field with $p$ elements and $F$ a field, algebraic over $\F_p$. Let $z$ be a variable, $\F_p[z]$ and $F[z]$ the rings of polynomials of $z$ over $\F_p$ and $F$ accordingly, and let $s_1,\dots,s_\nu$ be irreducible elements of $\F_p[z]$ which remain irreducible in $F[z]$. Let $S=\{ s_1,\dots ,s_\nu\}$ and $R=F[z,S^{-1}]$ be the ring which is generated over $F$ by $z$ and the inverses of the elements of $S$. 
Consider $\mathcal{R}$  as a structure (model) of the language $\mathcal {L}_p(z):= \{ =, +,x\mapsto x^p, 0,1, x\mapsto zx, \in F\}$ with symbols $+$ for addition, $x\mapsto x^p$ for the \emph{Frobenius map}, constant-symbols for $0$ and $1$,  the function symbol $x\mapsto zx$ for the multiplication-by-$z$ map  and a symbol for belonging to $F$.   
In \cite{robinson1951undecidable} R. Robinson proved that  the ring theory of rings of polynomials such as $\F_{p}[z]$, in the language of rings, augmented by a constant-symbol for $z$, is undecidable. 
In \cite{Denef1979} it was proved that even the positive-existential theory of a polynomial ring in this language is undecidable. 
The similar result for the rings $R$ was proved in  \cite{shlapentokh1993}.  

It is then natural to ask questions of decidability of substructures of the ring-structure of $R$. Here we prove the following.
 Consider  $\mathcal{L}_p :=\{ +, =, x\mapsto x^p, 0,1\}$ for the restriction of the language $\mathcal{L}_p(z)$, with interpretations as above.  
 Also consider the extension $\mathcal{L}_p(z)^e$ of $\mathcal{L}_p(z)$ by predicate symbols $P_\sigma$, one for each formula $\sigma$ of $\mathcal L_p$. We interpret $P_\sigma(\alpha)$ by `{\it each element of the tuple $\alpha$ is an element of $F$ and $\sigma (\alpha )$ holds true over $F$}' - we assume that all the free variables of the formula $\sigma$ are among the tuple of variables $\alpha$. We prove:
\begin{theorem}\label{main} 
Let $F$ be an algebraic field extension of $\F_p$. Let $\mathcal R$ be as above. Then the following hold:
\begin{enumerate}
\item The ${\mathcal{L}_p(z)^e}$-theory of $\mathcal R$ is effectively model-complete. 
\item Every ${\mathcal{L}_p(z)^e}$-sentence is equivalent in $\mathcal R$ to a sentence of the form $P_\sigma$, where $\sigma$ is a sentence of $\mathcal L_p$. 
\end{enumerate} 
\end{theorem}
We prove model-completeness by constructing an algorithm which converts 
any existential $\mathcal{L}_p(z)^e$-formula  to an equivalent, in  $\mathcal R$, universal $\mathcal{L}_p(z)^e$-formula. 
It is well known that model completeness of the theory of a countable model with a recursive elementary diagram  implies decidability of the theory (see current developments in \cite{ChubbMillerSolomon2021}), so we obtain:

\begin{corollary}
Assume that the $\mathcal L_p$-theory of the field $F$ is decidable. Then the $\mathcal{L}_p(z)^e$-theory of $R$ is decidable.
\end{corollary}
Item 2 of Theorem \ref{main}  says that `questions'  (first-order sentences) about $\mathcal R$ may be effectively translated into questions about $F$ 
(as a model of $\mathcal{L}_p$).
 This replies positively, for the structures $\mathcal{R}$, 
 to a `program' asked by Leonard Lipshitz: 
 {\it `Identify theories with universe a polynomial ring 
 $F[z]$ or a field of rational functions, 
 extending the structure of  addition by commonly used 
 operations and relations, which have the property that one 
 can effectively translate first-order sentences of the 
 structure into questions about $F$ and, 
 possibly, other elementary mathematical structures, e.g., 
 groups'}. This follows the type of results of J. Ax and S. Kochen for fields of $p$-adic numbers in \cite{AxKochen1965}.  

Theorem \ref{main} is, in part, a generalization of the results of \cite{PheidasZahidi2004}, where a similar theorem was proved for rings of polynomials (i.e., when $S$ is the empty set), for any perfect field $F$ (not necessarily algebraic). A similar result (model-completeness) was proved  in \cite{onay2018} by Onay for the henselization of $\F_{p}[z]$,  seen as a module over $\F_{p}[z]$ (and, more generally, over a finite field or an algebraic closure of it). The problem of whether the $\mathcal{L}_p(z)^e$-theory of the field $\F_p(z)$ or $\tilde \F_p(z)$ is model-complete, or even decidable, remains open. 

The structure of addition and the Frobenius map  is interesting, not only for its own sake, but 
also because it is connected to various important algebraic and logical problems. For example, 
the derivative of a function (polynomial, rational or power series) is positive-existentially definable 
in $\mathcal{L}_p(z)^e$ (see the Introduction of \cite{PheidasZahidi2004}). So, the structure of $R$ as a model of addition and differentiation is encodable in its $\mathcal{L}_p(z)$-structure.  

For surveys regarding the decidability properties of algebraic 
structures the reader may consult 
\cite{PheidasZahidi2000}, \cite{PheidasZahidi2008}, \cite{Poonen2008} and \cite{Koenigsmann2018}.  
For other decidability results for polynomial rings of positive characteristic the reader may consult  \cite{sirokofskich2010} - 
an analogue of results of A. Semenov for the natural numbers with addition and the set of powers of a fixed prime number - and \cite{Pheidas1985} where it is proved that the existential theory of addition and divisibility over a ring of polynomials with coefficients in an existentially decidable field is decidable. For the algebraic and model theoretic  properties of the Frobenius map see \cite{mazur1975}, \cite{ChatzidakisHrushovski1999} and \cite{Hrushovski-arx}. 

Although the general strategy of  proof of \cite{PheidasZahidi2004} 
works for the rings $R$, several of its components are quite  
difficult to adapt. 
We describe briefly the main two.
A polynomial  $f\in \F_p[z]$, in $m$ variables, is \emph{additive} 
if, for all $a$ and $b$ in $(\tilde F(z))^m$, it satisfies 
$f(a)+f(b)=f(a+b)$. Notice that the polynomial terms of the language 
$\mathcal{L}_p(z)$, with a zero constant term, are such additive polynomials. 
Here we will consider only additive polynomials with coefficients 
in $\F_p[z]$. 
Such an $f$ is called \emph{strongly normalized} if its 
coefficients are in $\F_p[z]$, the degrees of $f$ with respect to each of 
its variables is the same, $p^s$, for some $s\in \N\cup\{0\}$ and 
the degrees of its leading coefficients are pairwise inequivalent 
modulo $p^s$.   
We develop an algorithm by which questions regarding solvability of - arbitrary - additive polynomials are reduced to similar questions for strongly normalized polynomials. 
The first crucial property of a strongly normalized polynomial 
$f$ is that, for any given $y\in R$, the inverse image 
$\{ x\in R^m\ :\ f(x)=u\}$ has a bounded height. 
This is relatively easy for the case that $R$ is a polynomial ring 
and we know this is not true if $R$ is substituted by the field of 
all rational functions $\F_p(z)$. 
In order to prove it for the rings 
$R$ we use the notion of a `\emph{Hasse derivative}' and recent results 
on the relative Algebra (see Section \ref{sub:the_hasse_derivative}). So we prove:
\begin{theorem}\label{th:Manos}
Let $f$ be a strongly normalized additive polynomial of the variables $x=(x_1,\dots ,x_n)$. 
Then there is a recursive function  $h$ 
which, to each additive polynomial $f$ of the language $\mathcal{L}_p(z)$
and  each $\ell\in \N$ associates a non-negative integer $h(f,\ell)$ such that the height of each element of the set 
$\{ x\in F(z)^n\ :\ |f(x)| \leq \ell \}$ 
is less than or equal to $h(f,\ell)$.
\end{theorem}
A second point where the strategy of \cite{PheidasZahidi2004} needs significant adaptations is
where we need to prove that the image of a strongly normalized polynomial,
whose leading coefficients form a basis of $F(z)$ over $F(z^{p^s})$ 
(where $p^s$ is the degree of the polynomial),  is `almost all of $R$' (in the sense of Lemma~\ref{reduct}). 
Our method works only for algebraic fields $F$ - not for general 
perfect fields.  
Item 2 of Theorem \ref{main}, even for the polynomial case,  is obviously stronger than the results of \cite{PheidasZahidi2004} and its proof requires some improvements in the logical treatment of the subject - see  the proof of Section~\ref{sub:proof_of_theorem}.

In the rest of this section, we fix our notation and give a sketch of the proof. 
Section~\ref{sub:reduct_normal} contains some necessary algebraic results.
We prove Theorem~\ref{th:Manos} in Section~\ref{sub:the_hasse_derivative}.
and Theorem \ref{main} in Section \ref{sub:proof_of_theorem}. 
Some tedious elementary proofs are gathered in the Appendix.  


\subsection{Notation, Definitions and some elementary algebraic Facts}\label{sketch}

We use the following notations and definitions.

\begin{enumerate}
\item $\N$ is the set of natural numbers.
\item $p$ is a prime number, $\F_{p}$ is a field 
with $p$ elements, $\tilde{\F_p}$ is an algebraic closure of $\F_{p}$. 
\item $F$ is a field of characteristic $p$ such that
 $\F_p\subseteq F\subseteq \tilde \F_p$.
\item $z$ is a variable, $F[z]$ is the ring of polynomials of $z$ with coefficients in $F$ and $F(z)$ is its field of quotients, i.e., rational functions of $z$ with coefficients in $F$.
\item $S$ is a finite set $S=\{s_1,\dots, s_\nu\}$ of irreducible elements of $\F_p[z]$, which we assume to remain irreducible as elements of $F[z]$.
\item $R$ is the ring $F[z,s_1^{-1},\ldots, s_\nu^{-1}]$, i.e., the ring generated over $F[z]$ by the inverses of the elements of $S$.
 \item The language $\mathcal L_p$ is defined as $\mathcal{L}_p := \{ +, x \mapsto x^p,=, 0,1\}$. Its symbols are interpreted as follows: $+$ denotes addition, $x \mapsto x^p$ denotes the \emph{Frobenius map}, $=$ denotes 
 equality, $0$ and $1$ are constant symbols for the obvious elements of $\F_p$.  
 \item The language $\mathcal L_p(z)$ is defined by
 $\mathcal{L}_p(z) =\mathcal{L}_p \cup  \{\in F, 
 x \mapsto zx\}$  where $x \mapsto zx$ is the multiplication-by-$z$ map, and $\in F$ is a unary predicate denoting belonging  to $F$.
 \item If $\alpha =(\alpha_1,\dots ,\alpha_m)$ is a tuple of the variables $\alpha_i$, then $\alpha\in F$  denotes the relation $\wedge_{i=1}^m \alpha_i\in F$.  
 \item The language $\mathcal L_p(z)^e$ is the extension of 
 $\mathcal{L}_p(z)$ by the predicates $P_\sigma$, where 
 the index $\sigma$ ranges over the set of formulas 
 $\sigma(a_1,\dots,a_n)$ of $\mathcal{L}_p$. We assume that all the free variables of 
 $\sigma$ are among $a_1,\dots,a_n$. The predicate $P_\sigma$ 
 associated to $\sigma(a_1,\dots,a_n)$ is interpreted as 
 the $n$-ary relation ``$a_1,\dots,a_n \in F$ 
 and $\sigma(a_1,\dots,a_n)$ is true in $F$''.

 \item $\mathcal R$ denotes the model of the language 
 $\mathcal L_p(z)^e$, with universe $R$, with symbols interpreted as above.
 \item A variable ranging only over $F$ will be called an \emph{$F$-variable}.
\item  For $x\in F(z)$, if $Q$ is an irreducible element of $F[z]$, 
we write $ord_Q(x)$ for the \emph{order} of $x$ at $Q$. 
Recall that by convention $ord_Q (0)=\infty$. 
If $Q=z-\rho$, with $\rho\in F$, then we write $ord_\rho$ instead of $ord_Q$.
The \emph{order at infinity} of $x$, denoted $ord_\infty (x)$, 
is the degree of the denominator of $x$ minus the degree of its numerator. 
\item An \emph{additive polynomial} is a polynomial of the form 
\begin{equation}
  f(x)=\sum_{i=1}^n f_i(x_i),
\end{equation}
where $x=(x_1,\ldots,x_n)$ and, for each $i$, 
\begin{equation}
  f_i(x_i)= b_{i}x_i^{p^{s(i)}}+
\sum_{j=1}^{s(i)} c_{i,j}x_i^{p^{s(i)-j}},
\end{equation}
with $b_{i},c_{i,j}\in\F_p[z]$. The \emph{degree} of $f$ is $\deg(f):=\displaystyle\max_i\{p^{s(i)}\}$.
\item We say that the additive polynomial $f(x)$ is a polynomial 
\emph{of all the variables} of the tuple $x=(x_1,\dots ,x_m)$ whenever the degree of $f$ in each of the variables $x_i$ is positive, for $i=1,\dots, m$.
\item For $s\in\N$, let $\mathcal{V}_s(\F_p)$ be $\F_p(z)$ considered as a vector space over the field $\F_p(z^{p^s})$. 
Respectively,  let $\mathcal{V}_s(F)$  be $F(z)$
considered as a vector space over the field $F(z^{p^s})$. 
\item An additive polynomial $f$ is called 
\emph{normalized} if all degrees $s(i)$ are equal and the set of leading coefficients $\{b_i\ :\ i=1,\dots,n\}$ is linearly independent 
over $\mathcal{V}_s(\F_p)$. 
An additive polynomial is called \emph{$p$-basic} if it is normalized and the set $\{b_{1}, \dots, b_{n}\}$ forms a basis for $\mathcal{V}_s(\F_p)$.
Moreover, $f$ is \emph{strongly normalized} if $f$ is normalized and 
the degrees of the $b_i$ are pair-wise inequivalent modulo $p^s$, 
where $p^s$ is the degree of $f$.

\item
\begin{definition} 
\label{proper_transformation}
  A \emph{proper transformation} is a tuple 
  $\xi=(\xi_1$,$\dots$,$\xi_n)$ such that each $\xi_i(X,\beta)$ 
  is an additive polynomial of the variables 
  $x_1,\dots,x_m$, $\beta_1,\dots,\beta_{\mu}$ and such that the 
  map $\xi$ defined by\\ 
  \begin{tabular}{rccp{3cm}} 
  $\xi$: & $R^m \times F^\mu$ & $\rightarrow$ & $R^n$ \\
  & $(x_1,\dots,x_m,\beta_1,\dots,\beta_\mu)$ & $\mapsto$ & 
  $(\xi_1(x_1,\dots,x_m,\beta_1,\dots,\beta_\mu), \dots$, $ \xi_n(x_1,\dots,x_m,\beta_1,\dots,\beta_\mu))$\\
  \end{tabular}
  \\
  is surjective. Also note that the composition of proper 
  transformations is again a proper transformation.
\end{definition} 
We will user proper transformations in order to change variables.
\item Let $x$ be an tuple of variables and $f$ be an additive 
polynomial of the variables of $x$. 
When all the variables of the additive polynomial $H$ are 
$F$-variables then we write 
 $Im_{F}(H):=\{ y\in R\ |\ \exists \alpha\in F\ H(\alpha )= y\}$. 
We write $Im(f):=\{ y\in R\ |\ \exists x\in R\ f(x)=y\}$. 
\item A \emph{bounded term} is any expression of the form 
$\frac{1}{e}G(\alpha)$, where $e\in \F_p [z]$ and $G(\alpha )$ is 
an additive polynomial of the tuple of $F$-variables $\alpha$. 
Notice that a bounded term is not always a term of the language 
$\mathcal{L}_p(z)^e$ but we will be writing them with the 
understanding that `clearing of the denominators' is performed 
immediately - this makes sense because the only relation symbol 
of $\mathcal{L}_p(z)^e$, apart from $\in F$, is equality. 
We write $Im_F(\frac{1}{e}G)$ for $\frac{1}{e}Im_F(G)$.
\item The \emph{height} of a rational function $u=\frac{a}{b}\in F(z)$, where $a,b\in F[z]$ are coprime, is $|u| := \max\{ \deg(a),\deg (b)\}$. The height of $0$ is not defined.
\item The {\it partial fraction decomposition} of a rational function, from elementary algebra, is given by: 
\begin{fact}\label{partialFractions}
Let $x\in F(z)$ and let $Q_1,\dots ,Q_r$ be all the monic, irreducible polynomials of $F[z]$ at which $x$ has a pole. Then $x$ can be written  in a unique way as
\begin{align}
x=g(z)+\sum_{i,j} \frac{d_{i,j}(z)}{Q_i^j},
\end{align}
where the index $i$ ranges in the set $\{1,\dots ,r\}$ and, for each $i$, the index $j$  ranges in a non-empty subset of $\N$. Each of $g(z)$ and $d_{i,j}(z)$ is in $F[z]$ and for each $(i,j)$ the degree of the polynomial $d_{i,j}(z)$ is less than the degree of $Q_i$.   
\end{fact}
\end{enumerate}

\begin{fact}\label{ex}
(a) Let $u\in R$ and $c\in F[z]$ with $\deg (c)=d>0$. 
Then there is a $v\in R$ and a polynomial $r$, of degree 
less than $d$, such that $u=vc+r$. 
If no irreducible factor (over $F[z]$) of $c$ is invertible in $R$, 
then $v$ and $r$ are unique.     

(b) Given $u\in R$, $c\in F[z]\setminus \{ 0\}$ and $N\in \N$, there are $v\in R$ and $r_i\in F[z]$, for $i=0,\dots ,N$, with $deg (r_i)<deg(c)$ such that 
$$
u=r_0+r_1c+\dots +r_{N}c^N+vc^{N+1}\ .
$$

(c) Let $Q\in \F_p[z]$ which is not divisible by any polynomial that is
irreducible in $F[z]$ and 
invertible in $R$, of degree $d\geq 1$. Then, for any $x\in R$, the formula $x\not\in F$ is equivalent in $\mathcal R$ to 
\begin{gather}
\exists y \exists \alpha_0\dots \exists \alpha_{d-1} \\
[\alpha_0,\dots ,\alpha_{d-1}\in F\wedge x=yQ+\alpha_{d-1} z^{d-1}+\dots \alpha_1 z+\alpha _0\wedge (y\ne 0\vee \bigvee _{i=1}^{d-1} \alpha_i\ne 0)] \notag
\end{gather}
\end{fact}

The proof is elementary. For completeness, we include one  
in Section~\ref{sub:proof_of_fact}.
 
\begin{fact}\label{one_equation}
A system of equations may substituted by one equation due to the equivalence $(x=0\wedge y=0)\leftrightarrow x^p+zy^p=0$.
\end{fact}

\begin{fact}\label{onto1} 
Let $F$ be a perfect field. Let $s\in \N\cup\{0\}$. Then:
\begin{itemize}
\item The set $\{ z^i\ |\ 0\leq i\leq p^{s}-1\}$ is a basis of both $\mathcal V_s(\F_p)$ and  
$\mathcal V_s(F)$. 
\item Let  $B$ be a subset of $\F_p(z)$ which is linearly independent in 
$\mathcal V_s(\F_p)$. 
Then it is also linearly independent  in $\mathcal V_s(F)$. 
Consequently, any basis of $\mathcal V_s (\F_p)$ is also a 
basis of $\mathcal V_s (F)$. 
\end{itemize}
\end{fact} 


\subsection{Existential Formulas} 
\label{sub:existential_formulas}
Let $u$ be a term of $\mathcal{L}_p(z)^e$. From Fact~\ref{ex}, the formula  
$u \notin F$ can be substituted by the equivalent 
$\exists \alpha,x [\alpha \in F \wedge u= \alpha_0+\alpha_1z+\dots  +\alpha_{r-1}z^{r-1}+g x \  \wedge x \neq 0]$, where $g$ is an irreducible polynomial  in $F[z]$, of degree $r$,  which is not in $S$.
If $u$ is a term but not a variable, we can replace the 
formula $u \in F$ by the equivalent $\exists \alpha [\alpha \in F \wedge u= \alpha]$. Therefore, any formula is equivalent in $\mathcal R$ to a formula in which the negation of the predicate-symbol $\in F$ does not occur  and in which $\in F$ is applied only to variables. 
Such a formula of $\mathcal L_p(z)^e$, where all quantified variables 
range over $F$ will be called a \emph{bounded formula} (cf. \cite{PheidasZahidi2004}, p. 1021). 

\begin{fact}\label{boundedTerms} 
The set $\{ x\in R\backslash\{0\} \ :\ |x|\leq k\}\cup \{0\}$ is 
definable in $\mathcal L_p(z)$ by a \emph{bounded existential formula}, 
i.e., one which 
is existential and its quantified variables are all $F$-variables.
\end{fact}

We leave to the reader to verify that an existential formula of 
$\mathcal L_p(z)^e$  is equivalent in ${\mathcal R}$ to a disjunction 
of formulas of the form:
\begin{equation}\label{existential}
\phi (u,\{ v_j\}_{j\in J})\ :\ \exists x,\alpha\ 
[\alpha \in F \wedge  \psi (x,\alpha )]\ ,
 \end{equation}
 where 
\begin{equation}\label{psi}
 \psi(x,\alpha)\ :\  
 f(x) + H(\alpha) = u
  \wedge_{j \in J} e_j(x) + G_j(\alpha) \neq v_j \wedge P_\sigma (\alpha) \ .
\end{equation}
under the conventions:
\begin{itemize}
\item $x=(x_1,\dots ,x_m)$ is a tuple of the variables $x_i$.
\item   $\alpha$  is a tuple of $F$-variables, each of them distinct from each variable of $x$.
\item $f$ is an additive polynomial of all the variables of $x$ (hence, by convention, we have $deg_{x_i}(f)>0$ for every $i$), with coefficients in $\F_p[z]$.
\item Each $e_j$ is an additive polynomial of some of the variables of $x$. 
\item $H$ is an additive polynomial in some of the variables of $\alpha$.
\item Each $G_j$ is an additive polynomial  in some of the variables of $\alpha$. 
\item $u$ and the $v_j$ are terms of $\mathcal L_p(z)$.
\item No variables among those of $x$ or $\alpha$ occurs in $u$ or any of the $v_j$.
\item The predicate symbol $P_{\sigma}(\alpha )$ may have more variables than those of $\alpha$ occurring in it.
\end{itemize}



\subsection{Overview of this article.} 
\label{sub:paper_overview}

We start with a formula $\phi $ as above. 
In Section~\ref{sub:reduct_normal}, we show that, through a certain type of change of variables (see \emph{proper transformation} in Definition~\ref{proper_transformation}), 
one may assume that the  polynomial $f$ is strongly normalized
or the zero polynomial. The way to do this is similar to that of \cite{PheidasZahidi2004}, but several details have to 
be adjusted to the new rings $R$. 
In particular, Lemma~\ref{reduct} 
constitutes a new approach.
In Section~\ref{sub:the_hasse_derivative}, we prove  Theorem~\ref{th:Manos}. 
This allow us to substitute sets of the form 
$\{x\in F(z): |f(x)|\leq k\}$ by sets of the form
$\{x\in F(z): |x|\leq h\}$, where $k,h\in\N$.
Finally, in Section~\ref{sub:proof_of_theorem}, 
we use the previous steps in order to show that $\phi$ is equivalent 
to a universal formula.


\section{Strongly normalized polynomials and properties}
\label{sub:reduct_normal}


The following lemmas will be used in the proof of Theorem~\ref{th:Manos}, 
In this Section, we prove Theorem~\ref{th:Manos}.


\begin{lemma}[counterpart of \cite{PheidasZahidi2004}, Lemmas 3.3 and 3.4]
\label{lem:normalized-R}
Let $f$ be an additive polynomial in $m_0$ variables, 
with coefficients in $\F_p[z]$. 
Then, there is a proper transformation 
$\xi : R^{m} \times F^k \rightarrow R^{m_0}$, a strongly 
normalized additive polynomial $\tilde f$ in $m$ variables, 
with coefficients in $\F_p[z]$ and an additive polynomial $G$ 
in only $F$-variables, each one of them distinct from the 
variables of $\tilde f$, such that:
\begin{itemize}
\item $f \circ \xi=\tilde{f}+G$,
\item $m\le m_0$,
\item $\deg(\tilde{f}) \leq \deg(f)$,
\item $Im(f)=Im(\tilde{f})+Im_F(G)$,
\item $\xi$ and $ \tilde{f}$ are effectively computable from $f$.
\end{itemize}
\end{lemma}

The proof is similar to those of Lemma 3.3 and 3.4 of 
\cite{PheidasZahidi2004}. For completeness, we provide outlines 
in Section~\ref{sub:proof_of_lemma_lem:normalized-R}.

\begin{lemma}\label{basic}
For any strongly normalized additive polynomial $\tilde f$ in the 
variables of $x =(x_1,\dots, x_m)$, ranging over $R$, there is an 
additive polynomial $h$, in the variables of $v=(v_1,\dots, v_{p^s-m})$, also ranging over $R$, each one distinct from the variables of $x$, such that $\tilde f + h$ is  $p$-basic and strongly normalized. 
\end{lemma}

For a proof, see Lemma 3.5(a) of \cite{PheidasZahidi2004}.

\begin{definition}
Let $f$ an additive polynomial over $F(z)$, in $m$ variables. 
We define the relation $\sim_f$ over $F(z)$ by 
\begin{equation}
u\sim_f u' \mbox{ if and only if }  u'-u=f(\tilde{x})  
\text{, for some } \tilde{x}\in R^m.
\end{equation}  
\end{definition}

\begin{lemma}
\label{reduct} 
Assume that $F\subset \tilde \F_p$. Then there is a  primitive recursive function $E_{ord}$ from the set of $p$-basic polynomials of 
$\mathcal{L}_p(z)$  to $\N$ such that, for any $p$-basic polynomial 
$f$ over $\F_p[z]$ the following hold:
\begin{enumerate}
\item  For any  $u \in R$, there exists $u' \in R$ with the following properties:
\begin{itemize}
  \item $u \sim_f u'$.
  \item Each pole of $u'$ is a pole of $u$ - the word `pole' includes the pole at infinity.
  \item For all $\rho \in \tilde \F_p \cup \{ \infty\}$, we have 
  $|ord_\rho (u')| \leq E_{ord}(f)$. 
\end{itemize} 
\item Assume that $f$ has degree $p^s$, with $s\geq 0$. 
Write $N=E_{ord}(f)$, let $\alpha =(\alpha_0,\dots,$ $\alpha_{N-1})$ 
be a tuple of distinct $F$-variables and 
$G(\alpha):=\sum_{i=0}^{N-1} \alpha_i z^i$.  Then,
\begin{equation}
  R=Im(f)+ Im_F(\frac{1}{e^N}G),
\end{equation}
where  $e$ is the product of all elements of $S$.  
\end{enumerate}
\end{lemma}
\begin{proof} 
Consider an arbitrary $s\geq 0$, write $q=p^s$ and 
consider a $p$-basic polynomial  
$f(x)=\sum_{i=1}^{q}b_ix_i^q+g(x)$, of degree $q$, with coefficients in $\F_p[z]$, 
where $x=(x_1,\dots ,x_q)$ is the tuple of variables of $f$ and $g$ is an additive polynomial in $x$, of degree less than $q$, hence at most $q/p$. 

Since $f$ is $p$-basic, both $\{z^i\ :\ 0\leq i\leq q-1\}$ and $\{b_1,\dots ,b_q\}$ are bases of $\mathcal{V}_s(\F_p)$. 
Therefore, there is a non-singular matrix $A$ such that $B=AZ$, 
where  $Z$ and $B$ are the vectors of the $z^i$ and the $b_i$ 
respectively - in some order.
 Obviously $A$ has entries in $\F_p[z^q]$. 
Write $\Delta$ for the determinant of $A$. 
Using Cramer's Rule to solve the system  $B=AZ$ for the 
variable $z^i$, we obtain that, for any $i\in\{0,\dots ,q-1\}$, 
there are $e_{i,j}\in \F_p[z]$ such that 
\begin{equation}\label{eq:change_basis}
  \Delta z^i =\sum_{j=1}^q e_{i,j}^q b_j.
\end{equation}

We will now prove this auxiliary claim.
\begin{claim}\label{claim:D_divides}
Let $v=\frac{a}{Q^\ell}$, where $a,Q\in F[z]$, 
 $\ell\in \N$ and $Q$ is a monic, irreducible polynomial over $F[z]$.
 If $\Delta$ divides $a$ then there exists a $v'\in F[z]$ such that 
\begin{enumerate}\label{bounds3}
\item  $v\sim_f v'$,
\item the poles of $v'$ are roots of $Q$ and,
\item $|ord_Q ( v' )|\leq 
\frac{\ell+r}{p} =  \frac{1}{p}(|ord_Q(v)|+r)$, where $r$ is the least 
non-negative integer such that $p|\ell + r$.
\end{enumerate} 
\end{claim}

\begin{proof}
 Let $r$ be  the least non-negative integer such that $\ell+r=kq$ 
 for some $k\in\N$. Clearly $0\leq r<q$. 
 Since $\Delta$ divides $a$, we write $aQ^r=\Delta a'$ for some $a'\in F[z]$. 
 By Fact \ref{onto1}, there are $a_i\in F[z]$ so that 
 $$
 a' =\sum_{i=0}^{p^s-1} a_i ^qz^i.
 $$
 Then, using \ref{eq:change_basis}, there are $e_{i,j}\in F[z]$ such that, for each $i=0,\dots q-1$ we have 
 $\Delta z^i =\sum_{j=1}^q e_{i,j}^q b_j$,
 hence 
 $$
 v=\frac{\sum_{i=0}^{q-1}  a_i^q \Delta z^i }{Q^{kq}}= \frac{\sum_{i=0}^{q-1} \sum_{j=1}^q   a_i^q e_{i,j}^q b_j }{Q^{kq}}=\sum_{j=1}^q b_j  \tilde{x}_j^q,
 $$
 where 
 $\tilde x_j:=\frac{1}{Q^{k}}\sum_{i=0}^{q-1} a_i e_{i,j}$. Setting $\tilde x:=(\tilde x_1,\dots ,\tilde x_q)$, we have 
 $v=f(\tilde x)-g(\tilde x)$. Clearly, every $\tilde x_j$ is in $R$, 
 since $Q$ is an element of $S$ hence invertible in $R$. 
 Furthermore, observe that the order of $g(\tilde x)$ at $Q$ is at 
 least $-k\frac{q}{p}$, since the coefficients of $g$ are in 
 $\F_p[z]$.  So we set $v' = -g(\tilde x)$ and the proof is complete.
\end{proof}

 We now proceed in the proof of Item 1 of the Lemma. 
 Consider any $u\in F(z)$. 
 Write it as partial fractions, i.e., as sum of terms of the form 
 $u_{Q, \ell}=\frac{a}{Q^\ell}$ or $h$, where $a,h,Q\in F[z]$, 
 $\ell\in \N$ and $Q$ is a monic,  irreducible polynomial over $F[z]$. 
 For each term $u_{Q, \ell}$, we will find a  $u'_{Q, \ell}\in F(z)$, 
 whose poles (except the one at infinity) are roots of $Q$, with 
 order  bounded by $E_{ord}$, to be determined,
 such that 
 $$
 u_{Q, \ell}\sim_f u'_{Q, \ell}.
 $$
 Obviously $\sim_f$ is an equivalence and additive relation, 
 so $u$ will be related by $\sim_f$ to the sum 
 $\hat u=\sum _{Q,\ell} u'_{Q, \ell}+h$. 
 Finally, we  will treat the polynomial part $h$ of $\hat u$ in a 
 similar way, to obtain a $u'$ with the required properties.



Since $\Delta$ is in $\F_p[z]$, there is a finite extension 
$\F_{p^m}$ of $\F_p$ in which $\Delta$, as a polynomial in $z$, 
splits into linear factors (and $m\geq 1$). 
Then, the product of the distinct irreducible (over $\F_p[z]$) 
factors of $\Delta$ 
divides $z^{p^m}-z$, in $\F_p[z]$. 
Taking some high enough power $p^{m_0}$ of $p$, such that 
the multiplicity of each irreducible factors of $\Delta$ 
is $\leq  p^{m_0}$ we have that $\Delta$ divides 
$(z^{p^m}-z)^{p^{m_0}}= z^{p^{m+m_0}}-z^{p^{m_0}}$. 

 We rewrite the term $u_{Q,\ell}$ as
 \begin{equation}
   u_{Q,\ell}=\frac{a}{Q^{\ell}}=
  -\frac{a (Q^{p^{m+m_0}}-Q^{p^{m_0}})}{Q^{\ell+p^{m_0}}}+
  \frac{aQ^{p^{m}}}{Q^{\ell}}.
 \end{equation}
 
 Let $u_1:=-\frac{a (Q^{p^{m+m_0}}-Q^{p^{m_0}})}{Q^{\ell+p^{m_0}}}$ 
 and $u_2:=\frac{aQ^{p^{m}}}{Q^{\ell}}= Q^{p^{m}}u_{Q,\ell}$. 
 Since $m>0$, we have that the order of $u_2$ at $Q$ 
 is greater than the order of $u_{Q,\ell}$. 
 Moreover, the numerator of $u_1$ is divisible by $\Delta$, based on the 
 remarks of the previous paragraph. 
 Hence, by Claim~\ref{claim:D_divides}, there is a 
 $u'_1\in F[z,\frac{1}{Q}]$ such that $u_1\sim_f u'_1$ such that $|\ord_Q(u'_1)|\leq  \frac{1}{p}(\ell+p^{m_0}+r)$, 
 where $r$ is the least non-negative integer such that $p|\ell+p^{m_0}+r$. 

 Note that if $\frac{1}{p}(\ell+p^{m_0}+r)\geq \ell$ then 
 $p^{m_0}+r \geq (p-1)\ell$ and therefore, $\ell< \frac{p^{m_0}+q}{p-1}$,  
 since $r<q$. 
 Consequently, if  $|ord_Q(u_1)| = \ell\geq \frac{p^{m_0}+q}{p-1}$ 
 then $\ord_Q(u'_1)< -\ell =\ord_Q(u_1)$.
 In this case, we have $u_{Q,\ell}\sim_f u'_{Q,\ell}:= u'_1+u_2$ 
 and $|ord_Q(u'_{Q,\ell})|<|ord_Q(u_{Q,\ell})|$.
 
 After iterating this procedure for all terms $u_{Q,\ell}$ of $u$, 
 and adding the equivalency relations, we will end up with 
\begin{equation}
u = \sum_{Q,\ell}u_{Q,\ell} +  h\sim_f u' := \sum_{Q,\ell}u'_{Q,\ell} + h 
\end{equation}

Applying the partial fractions decomposition to $u'$, we obtain
\begin{equation}
u'= \sum_{Q,\ell}\hat{u}_{Q,\ell} + \hat{h} + h,
\end{equation}
where $\hat{u}_{Q,\ell}\in F(z)$ and $\hat{h}\in F[z]$. 
Using Lemma~3.5(d) of \cite{PheidasZahidi2004} for 
the polynomial $\omega:=\hat{h} + h$, there is a $\hat{\omega}\in F[z]$ such that 
$\omega\sim_f\hat{\omega}$ and $|\ord_{\infty}(\hat{\omega})|\le\Omega$, where
$\Omega$ is a bound depending only on $f$.  Define $E_{ord}$ to be 
\begin{equation}
E_{ord}(f):=\max\{\frac{p^{m_0}+q}{p-1},\Omega\},
\end{equation}

and Item 1 of the Lemma is proved. Item 2 follows clearly. 
\end{proof}

{\bf Note:} As seen by its proof, Lemma \ref{reduct} is true for 
any subring of $\tilde \F_p(z)$ in the place of $R$. 
Moreover, observe that the condition $F\subset \tilde \F_p$ seems to 
be essential.




\section{The order of poles of strongly normalized polynomials is bounded} 
\label{sub:the_hasse_derivative}

In this Section, we prove Theorem~\ref{th:Manos}. For 
the rest of this Section, $F$ is an 
algebraically closed field of characteristic $p>0$.

\subsection{Hasse derivatives and properties.} 
\label{sub:hasse_derivatives_and_properties_}

Hasse derivative or \emph{hyperderivative} \cite{Hasse:1936fw,schmidt1976equations} is a generalization 
of the derivative for fields of rational functions. It is especially
useful in positive characteristic. 
The $\epsilon-th$ Hasse derivative of $z^j$ is defined as
\begin{equation}\label{eq:hasse_definition}
  D_\epsilon(z^j) := \binom{j}{\epsilon} z^{j-\epsilon}, 
  \text{ for $j\geq 0$.}
\end{equation}
For $\epsilon=0$, $D_0$ is the identity function and 
we write $D$ instead of $D_1$.
In zero characteristic, it holds that  
\begin{equation}
  D_\epsilon(z^j) = \frac{1}{\epsilon!} \left(\frac{d}{dz}z^{j}\right).
\end{equation}

A set of useful properties is provided in 
\cite{Jeong:2011bh}. Some of these properties are - $f,g$ 
are rational functions of the variable $z$ over a field $F$:

\begin{description}
  \item[P1] The hyperderivative is linear, i.e., 
  \begin{equation}\label{eq:linearity_derivative}
    D_\epsilon(f+g) = D_\epsilon(f)+D_\epsilon(g), 
  \end{equation}  
  \item[P2] The hyperderivative satisfies the Leibniz product formula: 
  \begin{equation}\label{eq:product_rule}
    D_\epsilon(fg) = \sum_{i+j=\epsilon}D_i(f)D_j(g), 
  \end{equation}
  where $f,g\in F(z)$ and $i,j,\epsilon\geq 0$.
  \item[P3] If $p>0$ is the characteristic of $F$, 
  $m,\epsilon\in\N$ and 
  $f\in F(z)$, then
  \begin{equation}\label{eq:pth_power_derivative}
    D_\epsilon(f^{p^m}) = 
    \begin{cases}
    (D_j(f))^{p^m} &\mbox{ if } \epsilon=jp^m,\\
    0 &\mbox{ if } \epsilon\not\equiv 0 \mbox{ (mod } p^m).
    \end{cases}
  \end{equation}
  \item[P4] For $\epsilon\in\N$ and $0\neq f\in F(z)$,  
  \begin{equation}\label{eq:quotient_derivative}
    D_\epsilon\left(\frac{1}{f}\right) = \sum_{j=1}^{\epsilon} \frac{(-1)^j}{f^{j+1}} 
    \sum_{\substack{i_1,\ldots,i_j\geq 1\\i_1+\cdots+i_j=\epsilon}}
    D_{i_1}(f)\cdot\cdot\cdot D_{i_j}(f).
  \end{equation}
\end{description}

We will apply the above for a field $F$ of characteristic $p$. 
Let $q=p^s$ for some $s\in\N$. Let $\epsilon 
\in\{1,2,\ldots,p^s-1\}$. Since $\epsilon \not\equiv 0 \mod p^s$, 
we conclude that, for every $f\in F(z)$, we have
\begin{equation}\label{eq:derivation_constants}
  D_{\epsilon}(f^{p^s}) = 0, \quad \forall 
  \epsilon \in \{1,\dots, q-1\}.
\end{equation}

Let $\mathcal{B}=\{\beta_1,\dots,\beta_q\}$ be a basis of $F(z)$ 
over $F(z^{q})$ then, for every $g\in F(z)$, there exist 
$g_1,\dots, g_q \in F(z)$ such that
\begin{equation}\label{eq:projections_to_a_base} 
  g = g_1^q\beta_1 + \cdots + g_q^q\beta_q.
\end{equation}
Then due to \ref{eq:linearity_derivative}, 
\ref{eq:product_rule} and \ref{eq:derivation_constants}, it 
holds that 
\begin{equation}\label{eq:derivative_projection}
  D_\epsilon(g) = g_1^q D_\epsilon(\beta_1) + \cdots + 
  g_q^q D_\epsilon(\beta_q), 
  \quad \forall \epsilon \in \{1,\dots, q-1\}.
\end{equation}


\subsection{Linear independence criterion in positive characteristic} 
\label{sub:linear_independence_criterion_in_positive_characteristic}

Now, we present a theorem between Wronskians and linearly independent
sets of functions similar to the well known theorem of Calculus.
The corresponding theorem regarding the linear independence and the 
respective Wronskian (with hyperderivatives instead of classic ones),
for fields of positive characteristic, 
was initially described in \cite{Garcia:1987cs} and strengthened in 
\cite{Wang:1999ep}. 
Recall that we work over $F$ which is a field of characteristic $p>0$.

\begin{theorem}\label{th:linear_independence} [Th.1, \cite{Garcia:1987cs}] Let  $x_1,\dots,x_n\in F(z)$. Then 
$x_1,\dots,x_n$ are linearly independent over $\mathcal{V}_s(F)$ if and only if
there exist integers 
$0=\epsilon_1 <\epsilon_2 <\cdots <\epsilon_n < p^s$ with 
$\det\left( D_{\epsilon_i}(x_j) \right) \neq 0$.
\end{theorem}

Since $\mathcal{V}_s(F)$ is $F(z)$, seen as a vector space over 
$F(z^{p^s})$, it follows that:

\begin{corollary}\label{cor:basis}
The set 
$\{b_1,b_2,\dots,b_n\}\subset F(z)$ is linearly independent 
over $\mathcal{V}_s(F)$ 
if and only if there exist integers 
$0=\epsilon_1 <\epsilon_2 <\cdots <\epsilon_n < p^s$ with 
$\det\left( D_{\epsilon_i}(b_j) \right) \neq 0$.
\end{corollary}


\subsection{The orders of hyperderivatives} 
\label{sub:the_order_of_hyperderivatives}

Let $u\in F(z)\backslash\{0\}$. Fix $\lambda\in F$ and 
drop the subscript in $\ord_\lambda(u)$ for convenience.
We will give an upper bound of $\ord(D_\epsilon(u))$.
Initially, we consider 
the base 
$\{1,z-\lambda,(z-\lambda)^2,\ldots, (z-\lambda)^{q-1}\}$ of $F(z)$ over 
$F(z^q)$ and rewrite $u$ as
\begin{equation}\label{eq:original_u}
  u = u_0^q+(z-\lambda)u_1^q+(z-\lambda)^2u_2^q+\cdots+(z-\lambda)^{q-1}u_{q-1}^q,
\end{equation}
for some $u_0,u_1,\ldots$,$u_{q-1} \in F(z)$. 
Next, we apply the $\epsilon$-th hyperderivative, for 
$\epsilon \in \{1,\ldots$,$q-1\}$, and deduce that
\begin{equation}\label{eq:derivated_u}
  D_\epsilon(u) = u_0^q D_\epsilon(1) + 
  u_1^q D_\epsilon\left((z-\lambda)^1\right)+ \cdots + 
  u_{q-1}^q D_\epsilon\left((z-\lambda)^{q-1}\right).
\end{equation}  

It is now easy to see, using the composition properties of 
\cite{Jeong:2011bh}, that
\begin{equation}
D_\epsilon\left((z-\lambda)^k\right) = \binom{k}{\epsilon} (z-\lambda)^{k-\epsilon} 
  \label{eq:basis_element_derivative}.
\end{equation}
Observe that $D_\epsilon\left((z-\lambda)^k\right)$ vanishes if 
$k<\epsilon$.

Since $\ord\left(c(z-\lambda)^k\right)=k$, when $k\in\N$ and $c$ is a 
non-zero constant, for $k\geq\epsilon$,
Equation~\ref{eq:basis_element_derivative}, is 
equivalent to 
\begin{equation}
\ord\left(D_\epsilon\left((z-\lambda)^k\right)\right) =k-\epsilon.
\end{equation}

We are now ready to prove the following. 
\begin{theorem}\label{th:ord_inequality}
Let $u\in F(z)\backslash\{0\}$, where $F$ is a algebraically closed field 
of positive characteristic. If $\lambda\in F$, $\epsilon\in\N$ 
and $u,D_\epsilon(u)\neq 0$, then 
$\ord_\lambda\left(D_\epsilon(u)\right)\geq \ord_\lambda(u)-\epsilon$. 
\end{theorem}

\begin{proof}
We write $u$ 
 as in \ref{eq:original_u}, and observe that the 
terms on the right hand have pairwise distinct orders at $\lambda$,
\begin{equation}
  \ord\left((z-\lambda)^k(u_k)^q\right) = k + q\cdot \ord(u_k),
\end{equation}
since $0\leq k <q$. 
Denote by $\mu$ the remainder of dividing $\ord(u) $ by $q$. 
It follows that 
the term in $\{(z-\lambda)^k u_k^q: 0\leq k <q\}$ with the lowest 
order is $(z-\lambda)^\mu u_\mu^q$. 
Therefore, if we rewrite the terms 
$\{D_\epsilon\left((z-\lambda)^k\right)u_k^q: 0\leq k <q\}$ on the 
right hand of \ref{eq:derivated_u} as 
$\{\binom{k}{\epsilon}(z-\lambda)^{k-\epsilon}u_k^q: 
\epsilon\leq k <q\}$ (due to \ref{eq:basis_element_derivative}) 
we again notice that all non-vanishing terms have different 
orders 
again, equal to the order of the corresponding term of $u$ minus 
$\epsilon$.

If the term $D_\epsilon\left((z-\lambda)^\mu\right)u_\mu^q$
does not vanish, it will have the smallest order among all terms, 
and in this case 
\begin{equation}
\ord\left(D_\epsilon(u)\right) = 
\ord\left(D_\epsilon\left((z-\lambda)^\mu\right)u_\mu^q\right) =
\ord\left((z-\lambda)^{\mu} u_\mu^q\right)-\epsilon = 
\ord(u)-\epsilon
\end{equation}
If, however, that term vanishes, some other term with greater 
order (as a consequence of the analysis above) will give
$\ord\left(D_\epsilon(u)\right)$. In this case, 
$\ord\left(D_\epsilon(u)\right)>\ord(u)-\epsilon$ and 
our proof is complete.
\end{proof}


\subsection{The orders of inverse images and Proof of 
Theorem~\ref{th:Manos}} 
\label{sub:the_inverse_image_is_finite}

Let $y\in F(z)$ and let $x = (x_1,x_2,\ldots,x_n)$ be a solution of 

\begin{equation}\label{eq:non_homogeneous_equation}
y = f(x) = \sum_{i=1}^n f_i(x_i) = 
\sum_{i=1}^n b_i x_i^{p^s} + 
\sum_{i=1}^n\sum_{j=0}^{s-1} c_{i,j} x_i^{p^j},
\end{equation} 
where $x_i\in F(z)$, $n\leq q=p^s$, $s\in\N$, and $b_i,c_{i,j}$ 
belong to $F[z]$. 
We assume that the polynomials
$\{b_i: i=1,\ldots,q\}$ are linearly independent 
over $\mathcal{V}_s(F)$.
We will prove the following.

\begin{proposition}
\label{th:bounded_with_y}
Let $F$ and $f(\bar{X})$ as above and let 
$\bar{X} = (x_1,x_2,\ldots,x_n)$ a solution of $f(\bar{X})=y$, 
with $x_i\in F(z)$ for all $1\leq i \leq q$. 
Let $y=\frac{y_1}{y_2}$, where $y_1,y_2 \in F[z]$.
Then, there is a constant
$C$, 
 depending only on $\deg(y_2)$ and the 
coefficients of $f$, such that, 
for any pole $\lambda\in F$ of some $x_k$, 
for all $1\leq i \leq n$,
we have $\ord_\lambda(x_i)\geq C$.
\end{proposition}

\begin{proof}
Consider \ref{eq:non_homogeneous_equation} and 
move the terms of lower degree on the right hand side:
\begin{equation} \label{eq:homogeneous_separated_with_y}
b_1x_1^q + \cdots + b_nx_n^q  = 
u := -\sum_{i=1}^n \sum_{j=1}^s c_{i,j}x_i^{p^{s-j}}+y.
\end{equation}
By Corollary~\ref{cor:basis}, 
since $\{b_1,\ldots,b_n\}$ are linearly independent over 
$\mathcal{V}_s(F)$, 
there exist integers 
$0=\epsilon_1 <\epsilon_2 <\cdots <\epsilon_n < p^s$ with 
$\det\left( D_{\epsilon_i}\left(b_j\right) \right) \neq 0$. 

For every $\epsilon\in\{\epsilon_i: 1\leq i \leq n\}$, 
we apply the hyperderivative 
$D_{\epsilon}$ on both sides of \ref{eq:homogeneous_separated_with_y}.
We get  
\begin{equation}\label{eq:one_derivation_with_y}
\sum_{i=1}^n x_i^q D_{\epsilon}(b_i) = D_{\epsilon}(u) = 
\sum_{i=1}^n\sum_{j=1}^{s-1}\gamma_{i,j}x_i^{p^{s-j}} + 
\sum_{i=1}^n \gamma_{i}D_{\epsilon}(x_i)+D_{\epsilon}(y),
\end{equation}
where $\gamma_{i,j},\gamma_i$ all polynomials of 
$c_{i,j}$ and their hyperderivatives.  

Consider the system obtained by writing 
Equation~\ref{eq:one_derivation_with_y},
for $\epsilon=\epsilon_1,\epsilon_2,\dots,\epsilon_n$, 
\begin{align}
\sum_{i=1}^n x_i^q D_{\epsilon_1}(b_i) &= D_{\epsilon_1}(u)\notag\\
&\vdots \\
\sum_{i=1}^n x_i^q D_{\epsilon_n}(b_i) &= D_{\epsilon_n}(u)\notag
\end{align}
If we consider the variables $x_i^q$ on the left hand side 
of the system as the unknowns then the determinant of the system 
is $W:=\det\left( D_{\epsilon_i}(b_j) \right) \neq 0$.

Let $J$ be an index such that 
$\ord(x_J)\leq \ord(x_i)$ for every $i\in\{1,\dots,n\}$. 
Notice that $\ord(x_J)<0$ since we assumed that $\lambda$ is a 
pole of some $x_k$.
Applying Crammer's rule for the $J$-th index, we 
get that $Wx_J^q=\Lambda$, 
where
\begin{align}
\Lambda &:= 
\begin{vmatrix} 
b_1 & \cdots & b_{J-1} & u & b_{J+1} & \cdots & b_n \\
D_{\epsilon_2}(b_1) & \cdots &  D_{\epsilon_2}(b_{J-1}) & D_{\epsilon_2}(u) & D_{\epsilon_2}(b_{J+1}) & \cdots & D_{\epsilon_2}(b_n) \\
D_{\epsilon_3}(b_1) & \cdots &  D_{\epsilon_3}(b_{J-1}) & D_{\epsilon_3}(u) & D_{\epsilon_3}(b_{J+1}) & \cdots & D_{\epsilon_3}(b_n) \\
\vdots &  &  \vdots & \vdots & \vdots &  & \vdots \\
D_{\epsilon_n}(b_1) & \cdots &  D_{\epsilon_n}(b_{J-1}) & D_{\epsilon_n}(u) & D_{\epsilon_n}(b_{J+1}) & \cdots & D_{\epsilon_n}(b_n)
\end{vmatrix} \nonumber\\
& =  \sum_{i=1}^n\sum_{j=1}^{s-1}\delta_{i,j}x_i^{p^{s-j}} + 
\sum_{i=1}^n\sum_{k=1}^n \delta_{i}D_{\epsilon_k}(x_i)
+\sum_{k=1}^n \zeta_{k}D_{\epsilon_k}(y), %
\label{eq:FJ_right_hand_with_y}
\end{align}
where $\delta_{i,j},\delta_{i}$ and $\zeta_{k}$ are polynomials that
can be determined from $f$. 

We are now interested in determining  
bounds for  $\ord$ of each term on
the right hand of \ref{eq:FJ_right_hand_with_y}.

Regarding the terms of the form $\delta_{i,j}x_i^{p^{s-j}}$,
where $1\leq i \leq n$ and $1\leq j \leq s-1$, it holds 
that 
\begin{equation}
\ord\left(\delta_{i,j}x_i^{p^{s-j}}\right) 
\geq p^{s-1}\ord(x_J) + \Delta, 
\end{equation}
where $\Delta:= \min_{i,j}\{\ord(\delta_{i,j})\}$, since
$\ord(x_i)\geq \ord(x_J)$ and $\ord(x_J)<0$.

Regarding the terms of the form $\delta_{i}D_{\epsilon_k}(x_i)$,
where $1\leq i \leq n$ and $1\leq k \leq n$, it holds that
\begin{equation}
\ord\left(\delta_{i}D_{\epsilon_k}(x_i)\right)\geq 
p^{s-1}\ord(x_J) + E - (q-1),
\end{equation}
where $E:= \min_{i}\{\ord(\delta_{i})\}$. To 
prove this, we have used Theorem~\ref{th:ord_inequality} 
along 
with the relations $\epsilon_{k}\leq q-1$, $\ord(x_i)\geq \ord(x_J)$ and $\ord(x_J)<0$.

Since $\Delta,E\geq 0$, all terms
$\delta_{i,j}(x_i)^{p^{s-j}}$ and $\delta_{i}D_{\epsilon_k}(x_i)$
have $\ord_\lambda$ greater than or equal to 
$p^{s-1}\ord_\lambda(x_J) +1-q$. 
Let $\Phi_2$ be the order at $\lambda$ of the term $\sum_{k=1}^n \zeta_{i}D_{\epsilon_k}(y)$ 
appearing  in \ref{eq:FJ_right_hand_with_y}. 
Then, clearly, it holds that $\Phi_2\geq -\deg(y_2)-(q-1)$, due to  Theorem~\ref{th:ord_inequality}. 
Also observe that, from the relation $Wx_J^q=\Lambda$, it follows that
$\ord_\lambda(W)+q\ord_\lambda(x_J)=\ord_\lambda(\Lambda)$. 

We now compare $\Phi_2$ and $p^{s-1}\ord_\lambda(x_J) +1-q$ and 
distinguish two cases.

\begin{itemize}
\item If $p^{s-1}\ord_\lambda(x_J) +1-q< \Phi_2$, then
$\ord_\lambda(\Lambda)\geq p^{s-1}\ord_\lambda(x_J) + 1-q$.
In this case, we conclude that 
\begin{equation}\label{eq:33}
\ord_\lambda(x_J)\geq \frac{1-q-\ord_\lambda(W)}{q-p^{s-1}}\geq \frac{1-q-\deg_z(W)}{q-p^{s-1}}.
\end{equation}
\item  If  $p^{s-1}\ord_\lambda(x_J) +1-q \geq \Phi_2$, then
$\ord_\lambda(\Lambda)\geq \Phi_2\geq -\deg(y_2)-(q-1)$. 
It follows that 
\begin{equation}\label{eq:34}
\ord_\lambda(x_J)\geq \frac{1-q-\deg_z(W)-\deg(y_2)}{q}.
\end{equation}
\end{itemize}

Take $C$ to be the minimum of the bounds of $\ord_\lambda(x_J)$
in \ref{eq:33} and \ref{eq:34}. Note that $\ord_\lambda(x_i)\geq \ord_\lambda(x_J)\geq C$ for 
every $1\leq i \leq n$ and the proof is complete.
\end{proof}

\begin{proof}[Proof of Theorem~\ref{th:Manos}]
It suffices to establish the conclusion of 
Proposition~\ref{th:bounded_with_y}
not only for affine poles but also for the 
{\it pole at infinity}. 
Here is how: 
Let $\eta\in F$ such that none of the coefficients of $f$
has a zero at $\eta$.  
Apply the automorphism 
$z\mapsto 1/(z-\eta)$ and observe that coefficients 
$b_i$ map to 
$\tilde{b}_i/(z-\eta)^{\deg(b_i)}$, where $\tilde{b}_i\in F[z]$.
Let $M:=\max_i\{\deg(b_i)\}$. 
Clear the denominators by multiplying with $(z-\eta)^M$ and 
apply Proposition~\ref{th:bounded_with_y} for $\lambda=\eta$.
\end{proof}


\section{The Proof of Theorem \ref{main}}
\label{sub:proof_of_theorem}

We will now prove a series of propositions that will 
be used to prove Theorem~\ref{main}.

\begin{proposition}\label{logic1}
Let $f$ and $h$ be additive polynomials of degree $p^s$, in
$m$ and $n$ variables respectively,
where $m+n=p^s$. 
Let $H$ and $G$  be additive polynomials  in only
$F$-variables.
Assume that, for some $N\in \N$ and for some $e\in \F_p[z]$, it holds that
\begin{equation}
R=Im(f)+Im(h)+ Im_F(\frac{1}{e^N}G).  
\end{equation}

Then, the formula 
 $u\in Im(f)+Im_{F}(H)$ is equivalent, in $\mathcal R$, to the formula $\phi_1(u)$, defined as
\begin{equation}\label{2}
 \forall x=(x_1,\dots ,x_m)\ \forall y \forall \gamma\  
[(\gamma\in F\wedge u=f(x)+ h(y)+\frac{1}{e^N} G(\gamma))\rightarrow \pi_1 (y,\gamma )],
\end{equation}
where $\pi _1(y, \gamma )$ is the formula
\begin{equation}
\pi_1 (y,\gamma )\ :\  
\exists w \exists \alpha 
[\alpha\in F \wedge f(w)+h(y)+\frac{1}{e^N}\hat G(\gamma )=H(\alpha)]\ .
\end{equation}
In the writing of the formulas above  we mean that the variables of the tuples $x,y,w,\alpha,\beta$ and $\gamma$ are
pairwise distinct. 
\end{proposition}

Note that essentially $e = s_1s_2\cdots s_\nu$. 
The proof is similar to the one of Claim 4.2 (iii) of 
\cite{PheidasZahidi2004}.
For completeness, we include a detailed proof  
in Section~\ref{sub:proof_of_proposition_logic1}.


We will be referring to an existential formula $\phi (u,\{ v_j\}_{j\in J})$, 
as in ~\ref{existential}.

\begin{proposition}\label{logic2}
With notation as above, the formula  $\phi (u,\{ v_j\}_{j\in J})$ is equi-valent, in $\mathcal R$, to the formula $\phi_2$, given by
\begin{gather}\label{1}
\phi_2\ :\ u\in Im(f)+Im_{F}(H)\wedge  \\
\forall w\ \forall \beta\ [(\beta\in F\wedge  f(w)+H(\beta )=u )\rightarrow  \pi_2 (\{ v_j\}_{j\in J}, w, \beta )]\notag,
\end{gather}

where 
\begin{gather}
\pi_2(\{ v_j\}_{j\in J}, w, \beta )\ :\  \exists t\exists \gamma\ [\gamma\in F\wedge f(t)+H(\gamma )=0 \ \wedge  \\
\{\wedge_{j \in J} \ e_j(t) + G_j(\gamma) \neq v_j -e_j(w)-G_j(\beta ) \} \wedge  P_{\sigma}(\beta +\gamma)]\notag.
\end{gather}
Here, $P_\sigma (\beta +\gamma )$ denotes the result of substitution in $P_\sigma (\alpha)$ of the tuple of variables 
$\alpha$ by the array of variables $\beta +\gamma$ where $+$ implies component-wise addition. The variables of the tuples $\alpha$,  $\beta$ and $\gamma$ are
pairwise distinct.
\end{proposition}

The proof is similar to the one of Claim 4.1 of \cite{PheidasZahidi2004}. 
For  completeness, we include  a detailed  proof in 
Section~\ref{sub:proof_of_proposition_logic2}. 
Note that Proposition~\ref{logic2} holds even if the ring $R$ is replaced by any subring of $F(z)$ that contains $F[z]$.

\begin{proposition}\label{universal}
A bounded existential $\mathcal{L}_p(z)^e$-formula is equivalent in $\mathcal R$ to a universal $\mathcal{L}_p(z)^e$-formula. 
\end{proposition}
\begin{proof} 
We will prove the proposition for any formula of the form
\begin{equation}
\pi (\{v_j\}_{j\in J}, \beta )\ :\ \exists \gamma\ [\gamma\in F\wedge  \{
\wedge_{j \in J} \  G_j(\beta ,\gamma) \neq v_j \} \wedge  P_{\sigma}(\beta ,\gamma)],  
\end{equation}
where $v_j$ are terms of $\mathcal L_p(z)$  and the $G_j$ are additive polynomials in the variables of the tuples of $F$-variables $\beta $ and $\gamma$. 
Let $M$ be the maximum of the degrees of the $G_j$ with respect 
to the variable $z$. 
Pick an element $Q$ of $\F_p[z]$, 
which is not divisible by any polynomial that is
irreducible in $F[z]$ and 
invertible in $R$. 
 For each of the terms $v_j$, we construct the term 
\begin{equation}
t_j := \sum_{i=0}^M \mu _{i,j}Q^i,
\end{equation}
where each $\mu_{i,j}$ is a term of the form 
$\mu_{i,j}=\sum_{k=0}^d \mu _{i,j,k}z^k$ with 
$d:=\deg(Q)$, and $\mu_{i,j,k}$ are new $F$-variables.  
For each $j\in J$, let $y_j$ be a new variable (ranging over $R$). 

Observe that, by Fact~\ref{ex}, for each $j\in J$ and any value of the term $v_j$ over $R$, 
there are values of the variables of $\mu_{i,j}$ over $F$ so that $v_j=t_j+Q^{N+1}y_j$ holds true. And, for those values and any value of the variables of the tuples $\beta $ and $\gamma$, the sub-formula $G_j(\beta ,\gamma) \neq v_j$ is equivalent to $y_j\ne 0\vee G_j\ne t_j$. Hence $\pi$ is equivalent to  the formula 
\begin{align}
\chi (\{v_j\}_{j\in J}, \beta )\ :\ \forall \mu \forall y\ (\beta\in F\wedge \mu\in F\wedge \{\wedge_{j\in J} \ v_j=t_j+Q^{N+1}y_j \})\rightarrow \\ (\vee _{K\subset J} ( \{\wedge_{j \in K}\ y_j\ne 0\} \wedge \{\wedge _{j \in J\setminus  K}\ y_j = 0 \} \wedge  \pi _{K} [v_j/t_j])),   \notag
\end{align}
 
where 
\begin{itemize}
\item $\mu$ stands for the tuple of all variables $\mu_{i,j,k}$, 
$y$ stands for the tuple of variables $y_j$,
\item the index $K$ of $\vee_K$ ranges over all subsets of the set $J$ of indices in $\pi$ and,
\item for each subset $K$ of $J$, the formula $\pi_K  [ v_j/t_j]$ 
is the formula that results from $\pi $ by deleting the 
inequalities $G_j(\beta ,\gamma) \neq v_j$ for which  $j\in K$ 
and replacing  each  term $v_j$, for which $j\in J\setminus K$, 
by $t_j$. 
\end{itemize}

We prove the equivalence of $\pi$ and $\chi$: 

Assume that $\pi (\{v_j\}_{j\in J}, \beta )$ is true for some set of values $\tilde v_j$ of the terms $v_j$ and $\tilde \beta $ of the $F$-variables $\beta$. Write each $\tilde v_j$ as 
$\tilde v_j=\tilde t_j+Q^{N+1}\tilde y_j$, using Fact~\ref{ex}. 
Let $K$ be the set of indices $j$ for which $\tilde y_j\ne 0$. Then, for each $j\in J\setminus K$, we have $\tilde y_j =0$. 
Then, since $\pi (\{\tilde v_j\}_j, \tilde \beta )$ is true, 
it holds that there are $\tilde \gamma$ over $F$ so that the 
inequalities $G_j(\tilde \beta ,\tilde \gamma) \ne \tilde v_j$ 
hold for $j\in J\setminus K$, hence, for those $j$, we have $\tilde{y}_j=0$ 
and  $G_j(\tilde \beta ,\tilde \gamma) \neq \tilde t_j$ holds true 
and $P_{\sigma}(\beta ,\gamma)$ is true; Hence
$\pi _{K} [ v_j/t_j]$ is true. 
So $\chi (\{\tilde v_j\}_j, \tilde \beta )$ is true. We leave the converse to the reader. 
Now observe that each sub-formula  $\pi _{K} [ v_j/t_j]$ of $\pi$ is equivalent in $\mathcal R$ to a formula of the form $P_\tau$, for 
some formula $\tau$ of $\mathcal{L}_p$. 
The same holds true for equations of the form $H(\beta )=0$ that may appear in the quantifier free part of $\pi$. So the result holds in the desired generality.
\end{proof}


\begin{proof}[of Theorem~\ref{main}]

By Lemma~\ref{lem:normalized-R}, given an additive polynomial $f$
as in Proposition~\ref{logic1} or \ref{logic2}, there is a 
proper transformation $\xi ( Y, \delta )$ such that, setting 
$\tilde f (y)+G(\delta) =f(\xi (Y, \delta ))$, 
the additive polynomial $\tilde{f}$ is strongly normalized. 
Hence, since a proper transformation is onto, we may assume 
that the $f$ appearing in Propositions~\ref{logic1} 
or \ref{logic2} are strongly normalized.

Lemma~\ref{reduct} implies that, since the additive polynomial 
$f$  is strongly normalized, there are $h$ and $G$ so that the 
assumption of Proposition~\ref{logic1} is satisfied for 
$e\in \F_p[z]$ and $N\in\N$. 
Observe that, due to Theorem~\ref{th:Manos}, the variables of $w$ in 
$\pi_1$ can be substituted by a finite tuple of $F$-variables. 
Hence, formula $\pi_1(y, \gamma )$ is equivalent to 
a formula of the form $P_\tau$, for 
some formula $\tau$ of $\mathcal{L}_p$.  
It follows that $\phi_1(u)$ is equivalent to a universal 
$\mathcal L_p(z)^e$-formula.    

In Proposition \ref{logic2} observe that, from Theorem~\ref{th:Manos} 
and the equation
$f(t)+H(\gamma )=0$ in $\pi_2$, the variable $t$ may be 
substituted by a finite tuple of $F$-variables. Hence the 
formula $\pi_2(\{ v_j\}_{j\in J}, w, \beta )$ is equivalent 
in $\mathcal R$ to a bounded existential formula.

So far we have proved that any existential $\mathcal{L}_p(z)^e$-formula is equivalent in $\mathcal R$ to a disjunction of bounded existential formulas, in which equations do not occur. 
From Proposition~\ref{universal} it follows that formulas $\pi_1$ of Proposition \ref{logic1} and $\pi_2$ of Proposition \ref{logic2} are equivalent to universal $\mathcal{L}_p(z)^e$-formulas, and actually to ones with no more free variables than those present in $\pi_1$ and $\pi_2$ respectively.  

By induction on the number of alterations of quantifiers of an arbitrary $\mathcal{L}_p(z)^e$-formula in prenex form Theorem \ref{main} is proved.

For Item 2 of Theorem \ref{main}, it follows from Item 1 that any $\mathcal{L}_p(z)^e$-sentence is equivalent to an existential formula like $\phi$ in \ref{existential}, which is, in addition, a sentence. 
This means that, in this case,  the terms $u$ and $v_j$ are concrete elements of $\F_p[z]$. By Lemma   \ref{lem:normalized-R} we may assume, without loss of generality, that the additive polynomial $f$ is strongly normalized. Re-enumerate the variables of $x$ so that 
$x=(x_1,\dots ,x_k,x_{k+1},\dots ,x_m)$ and $x_{k+1},\dots,x_m$ are exactly the variables of $x$ which occur in $f$ with non-zero highest degree coefficient. 
Then, by Theorem~\ref{th:Manos},  for any value $\tilde x$ of 
the tuple $x$ which is a  solution of the equation  $f+H=u$, 
the heights of $\tilde{x}_{k+1},\dots ,\tilde{x}_m$ are effectively bounded, hence, the variables $x_{k+1},\dots ,x_m$ may be substituted by
 (existentially quantified) $F$-variables. 
 Therefore, we may assume that the sentence $\phi$ has no equations and amounts to the solvability of the system of inequalities $g_j+G_j\ne v_j$, 
 together with $P_\sigma$. 
 Clearly, because $R$ is an infinite domain, 
 all inequalities in which some of the variables 
 $x_1,\dots ,x_k$ occurs with a non-zero 
 coefficient may be satisfied simultaneously. All the inequalities in which none of the variables $x_1,\dots ,x_k$ occurs is clearly equivalent to 
 a formula of the form $P_\tau$. 
 Hence $\phi$ is equivalent in $\mathcal R$ to 
 a formula of the form $P_\tau$, for 
some sentence $\tau$ of $\mathcal{L}_p$. 

\end{proof}

\color{black}


\section*{Acknowledgments}
This research work was supported from Greek and 
European Union resources, through the National Strategic Reference Framework (NSRF 2014-2020), under the call ``Support for researchers with emphasis on young researchers (EDBM103)'' and the funded project ``Problems of Diophantine Nature in Logic and Number Theory''  with code MIS 5048407.

\printbibliography


\newpage
\appendix
\section{Appendix - Omitted Proofs} 
\label{sec:omitted_proofs}

\subsection{Proof of Fact~\ref{ex}} 
\label{sub:proof_of_fact}

\begin{proof}
To prove Item 1, consider a $u=\frac{a}{b}\in R$ with $a$ and $b$ coprime polynomials of $F[z]$ and let $c\in F[z]\setminus \{0\}$. Then $b$ is invertible in $R$. By the elementary algebra of polynomials we know that there are $v',r'\in F[z]$, with $deg(v')<deg (b)$ and $deg (r')<deg (c)$, such that $v'c+r'b=1$. So $\frac{a}{b}=v'\frac{a}{b}c+ar'$. Divide  the polynomial $ar'$ by $c$ - using euclidean division of polynomials - to obtain the result. 

For the uniqueness statement, assume in the above that $c$ is not divisible by any irreducible factor in $F[z]$ which is invertible in $R$. Assume that there are $v$, $\hat v$ in $R$ and $r$ and $\hat r$ in $F[z]$ such that $u=vc+r=\hat vc+\hat r$. Then $0=(v-\hat v)c+(r-\hat r)$ with $deg (r-\hat r)<deg (c)$. Then there is a $\hat b\in F[z]$, which is invertible in $R$, such that $\hat b (v-\hat v)\in F[z]$. Hence $0=\hat b (v-\hat v)c+\hat b (r-\hat r)$. Hence $c$ divides $\hat b (r-\hat r)$ in $F[z]$. But by hypothesis $c$ is co-prime to $\hat b$. Hence $c$ divides $r-\hat r$ in $F[z]$, so, necessarily $r-\hat r=0$. 

For Item 2, iterate the conclusion of Item 1 $N$ times, each time applying it to the `quotient' $v$ of the previous step.

Item 3 follows from Item 1.
\end{proof}

\subsection{Proof of Proposition~\ref{logic1}} 
\label{sub:proof_of_proposition_logic1}
\begin{proof} 
($\rightarrow$) Say that for some $\tilde u\in R$, for some value $\tilde x$ of the tuple of variables $x$ and for some value $\tilde \alpha $ (over $F$) of the tuple of variables $\alpha$ we have $\tilde u=f(\tilde x)+H(\alpha )$. 

Consider any value $\tilde {\underline{x}}$ of the tuple $\underline{x}$, any value $\tilde y$ of the tuple $y$ and any value (over $F$) of the tuple of variables $\tilde \gamma $ of the tuple $\gamma$ for which 

$\tilde u=f(\tilde {\underline{x}})+ h(\tilde y)+\frac{1}{e^N} G(\tilde \gamma )$.  
Set $\tilde w=\tilde {\underline{x}}-\tilde x$ where $-$ denotes component-wise subtraction. Then, by the additivity of $f$ we have 
\begin{gather}
f(\tilde w)+h(\tilde y) +\frac{1}{e^N} G(\tilde \gamma )=
f(\tilde {\underline{x}}-\tilde x)+h(\tilde y) +\frac{1}{e^N} G(\tilde \gamma )= \\
f(\tilde {\underline{x}})-f(\tilde x)+h(\tilde y) +
 \frac{1}{e^N} G(\tilde \gamma )=
\tilde u- f(\tilde x)=H(\tilde \alpha)\ .\notag
\end{gather}

($\leftarrow$) By hypothesis there are $\tilde {\underline{x}}$, $\tilde y$ and $\tilde \gamma$ so that 
$u=f(\tilde {\underline{x}})+h(\tilde y)+\frac{1}{e^N} G(\tilde \gamma )$.  Let $\tilde u\in R$ and assume that $\phi_3 (\tilde u)$.
Then  there is a $\tilde w$ and a $\tilde \alpha$ such that  
$f(\tilde w)+h(\tilde y)+\frac{1}{e^N} G(\tilde \gamma )=H(\tilde \alpha )$. 

Set $\tilde x=\tilde {\underline{x}}-\tilde w$. By the additivity of $f$ and we have 
\begin{gather}
f(\tilde x)+H(\tilde \alpha )=f(\tilde {\underline{x}})-f(\tilde w)+H(\tilde \alpha)= \\
[\tilde u-h(\tilde y)-\frac{1}{e^N} G(\tilde \gamma )]-
[H(\tilde \alpha )-h(\tilde y)-\frac{1}{e^N}G(\tilde \gamma )]+H(\tilde \alpha )=\tilde u\ , \notag
\end{gather}

hence $\tilde u\in Im(f)+Im_{F}(H)$.
\end{proof}
\color{black}

\subsection{Proof of Proposition~\ref{logic2}} 
\label{sub:proof_of_proposition_logic2}

\begin{proof}
Assume that $\tilde u$ and $\tilde v_j$ are  given values of the terms $u$ and $v_j$, respectively. 
Assume that for some value $\tilde x$ of the array of variables $x$ over $R$ and for some value $\tilde \alpha$  of the array of variables 
$\alpha$ over $F$ the statement $\psi (\tilde x,\tilde \alpha )$
is true in $\mathcal R$, with $\psi$ as in \ref{psi}, 
i.e.,
\begin{equation}
   f(\tilde x) + H(\tilde \alpha) = \tilde u \wedge_{j \in J} e_j(\tilde x) + G_j(\tilde \alpha) \neq \tilde v_j
 \wedge P_\sigma (\tilde \alpha ) 
\end{equation}
holds.

Let $\tilde w$ be a tuple of elements of $R$ and $\tilde \beta$ be a tuple of elements of $F$ such that $ f(\tilde w)+H(\tilde \beta)=\tilde u $
is true in $\mathcal R$. 

Define $\tilde t=\tilde x-\tilde w$ and $\tilde \gamma =\tilde \alpha -\tilde \beta$, where $-$ denotes component-wise subtraction. 
Then, by the additivity of $f$, $H$, $e_j$ and $G_j$ we have 
\begin{equation}
f(\tilde t)+H(\tilde \gamma )=f(\tilde x)- f(\tilde w)+H(\tilde \alpha )-H(\tilde \beta)=\tilde u-\tilde u =0,
\end{equation}
and for each $j\in J$ 
the following holds: 
\begin{equation}
e_j(\tilde t)+G_j(\tilde \gamma )=  e_j(\tilde x)+G_j(\tilde \alpha )-e_j(\tilde w)-G_j(\tilde \beta )\ne 
\tilde  v_j- e_j(\tilde w)-G_j(\tilde \beta )\ .
\end{equation}

Moreover $P_\sigma (\tilde \alpha)$ holds true, hence $P_\sigma (\tilde \beta +\tilde \gamma )$ holds true. 
 It follows that ${\mathcal R}\models \phi \rightarrow \phi_2$.

Now assume that $\phi_2$ is true in $\mathcal R$ for the given values $\tilde u$ and $\tilde v_j$ of $u$ and $v_j$, respectively.  Since $\tilde u\in Im(f)+Im_{F}(H)$ there are values $\tilde w$ over $R$ and $\tilde \beta $ 
over $F$ of the variables $w$ and $\beta$ such that 
$\tilde \beta\in F\wedge f(w)+H(\tilde \beta )=\tilde u$. 
Since $\phi_2$ is true, there are values $\tilde t$ over $R$ and $\tilde  \gamma$ over $F$ of the variables $t$ and $\gamma$, respectively, so that 
$f(\tilde t)+H(\tilde \gamma )=0\wedge_{j \in J} e_j(\tilde t) + G_j(\tilde \gamma) \neq v_j -e_j(\tilde w)-G_j(\tilde\beta ) \wedge P_\sigma(\tilde \beta +\tilde \gamma)$. 
Define $\tilde x=\tilde w +\tilde t$ and $\tilde \alpha =\tilde \beta +\tilde \gamma $. Obviously $\psi (x,\alpha )$ of \ref{psi} is true for the values  $\tilde x$ over $R$ of $x$ and $\tilde  \alpha$ over $F$ of $\alpha$. 
It follows that ${\mathcal R}\models \phi_2 \rightarrow \phi$.
 
\end{proof}


\subsection{Outline of the proof of Lemma~\ref{lem:normalized-R}} 
\label{sub:proof_of_lemma_lem:normalized-R}

\begin{proof}
We present the proof of the statement of the Lemma up to the point that the resulting 
  $\tilde f$ is $p$-free. 
  Consider a given additive polynomial $f(x)=\sum_{i=1}^{m_0}f_i(x_i)$ of degree $p^s$ where  each $f_i(x_i)$ is an additive polynomial of the variable $x_i$ only and of degree $p^{s_i}$  and coefficients $b_i\in \F_p[z]$. Assume that the set $B=\{ b_{i}z^{jp^{s_i}}\ |\ i=1,\dots ,m_0,\ 0\leq j< p^{s-s_i}\}$ is linearly dependent  in ${\mathcal V}_s(\F_p)$. Then there are $c_{i,j}\in \F_p[z]$, not all equal to $0$, so that 
  $\sum_{i,j}c_{i,j}^{p^s}b_iz^{jp^{s_i}}=0$, where the index $j$ ranges as in the set $B$. Re-enumerating the indices $i$ we may assume that some $c_{1,j}$ is not equal to $0$ and for each $i$ for which there is a $j$ so that $c_{i,j}\ne 0$ we have $s_1\geq s_i$. Let 
  \begin{equation}
    c:=\sum_{\mu =0}^{p^{s-s_1}-1} c_{1,\mu}^{p^{s-s_1}}z^{\mu}.
  \end{equation}
  Notice that, since the set $\{ 1,\dots ,z^{p^{s-s_1}-1}\}$ is linearly independent in $\mathcal V_{s-s_1} (\F_p)$  (see Fact \ref{onto1})
  we have $c\ne 0$. We apply the proper transformation:
  $ x_1=cy_1+H$ and, for $i\ne 1$,
  \begin{equation}
  x_i=y_i+\sum_{j=0}^{p^{s-s_i}-1} c_{i,j}^{p^{s-s_i}}z^{j}y_1^{p^{s_1-s_i}},
  \end{equation}
  where  $H=\alpha_0+\alpha_1 z+\dots +\alpha_{\ell-1} z^{\ell -1}$, $\ell$ is the degree of $c$ 
 and the $\alpha_k$ are new and pairwise distinct $F$-variables. 
  The polynomial that results from the transformation has the form $\hat f+G$, where $\hat f$ has variables the unrestricted (i.e. not $F$-variables) $y_i$  and $G$ is an additive polynomial in the $F$-variables $\alpha_k$.  We observe that the coefficient of $y_1^{p^{s_1}}$ in $\hat f$ is 
  $b_1c^{p^{s_1}}+\sum_{i\ne 1}\sum_{j=0}^{p^{s-s_i}-1} c_{i,j}^{p^{s}}z^{jp^{s_i}}=0$.
  Then $\hat f$  has degree in $y_1$ less than the degree of $f$ in $x_1$ and for all $i>1$ it has degree in $y_i$ equal to the degree of $f$ in $x_i$. Work by induction on the sum of the degrees of $f$ in each of its (unrestricted) variables. 
  The proper transformation $\xi$ is onto, by Facts \ref{ex} 
  and \ref{onto1}, so $Im(f)=Im(\hat f)+Im_F(G)$. 
  
  Now consider a given additive polynomial $f$, as above, which is $p$-free. It is easy to see that substituting each variable $x_i$ by 
  \begin{equation}
    x_{i,0}^{p^{s-s_i}}+\dots +z^k x_{i,k}^{p^{s-s_i}}+ z^{p^{s-s_i}-1} x_{i,{p^{s-s_i}}-1}^{p^{s-s_i}}\ ,
  \end{equation}
  where the $x_{i,k}$ are new variables, results in a normalized additive polynomial $\tilde f$: The degree of $\tilde f$ with respect to each of its variables is $s$ and the set of its leading coefficients is linearly independent in $\mathcal V_s(\F_p)$. By Fact  \ref{onto1}, $f$ and $\tilde f$ have the same image.
  
 To convert a normalized additive polynomial to a strongly normalized one takes the application of a sequence of  proper transformations - it is left to the reader or see \cite{PheidasZahidi2004}, Lemma 3.3.
\end{proof}



\end{document}